\documentclass[12pt,leqno,twoside]{amsart}

\usepackage{amssymb,amsmath,amsthm,soul,color}

\usepackage[noadjust]{cite}

\usepackage{t1enc}

\usepackage[utf8]{inputenc}
\usepackage{indentfirst}

\usepackage{todonotes}

\usepackage[left=2.5cm, right=2.5cm, top=2cm, bottom=2cm]{geometry}

\usepackage{url}

\usepackage{mathrsfs}

\usepackage{enumitem,graphicx}
\usepackage[colorlinks=true,urlcolor=blue,
citecolor=red,linkcolor=blue,linktocpage,pdfpagelabels,
bookmarksnumbered,bookmarksopen]{hyperref}
\usepackage[metapost]{mfpic}
\usepackage[normalem]{ulem}

\usepackage{lmodern} 

\newtheorem{Th}{Theorem}[section]
\newtheorem{Prop}[Th]{Proposition}
\newtheorem{Lem}[Th]{Lemma}

\newtheorem{Rem}[Th]{Remark}

\newcommand{\eps}{\varepsilon}

\newcommand{\R}{\mathbb{R}}

\newcommand{\Z}{\mathbb{Z}}

\newcommand{\cC}{{\mathcal C}}
\newcommand{\cD}{{\mathcal D}}

\newcommand{\cJ}{{\mathcal J}}

\newcommand{\cL}{{\mathcal L}}

\newcommand{\cP}{{\mathcal P}}

\newcommand{\cS}{{\mathcal S}}

\newcommand{\Ga}{\Gamma}

\newcommand{\weakto}{\rightharpoonup}

\numberwithin{equation}{section}

\begin{document}

\title[Homoclinic solutions to nonautonomous HS with sign-changing nonlinear part]{Note on homoclinic solutions to nonautonomous Hamiltonian systems with sign-changing nonlinear part}

\author[F. Bernini]{Federico Bernini}
	\address[F. Bernini]{
	\newline\indent Dipartimento di Matematica ``Federigo Enriques'',
	\newline\indent Università degli Studi di Milano,
	\newline\indent via C. Saldini 50, 20133 Milano, Italia}	
	\email{\href{mailto:Federico.Bernini@unimi.it}{Federico.Bernini@unimi.it}}	

\author[B. Bieganowski]{Bartosz Bieganowski}
	\address[B. Bieganowski]{
	\newline\indent Faculty of Mathematics, Informatics and Mechanics,
	\newline\indent University of Warsaw,
	\newline\indent ul. Banacha 2, 02-097 Warsaw, Poland}	
	\email{\href{mailto:bartoszb@mimuw.edu.pl}{bartoszb@mimuw.edu.pl}}	
	
\author[D. Strzelecki]{Daniel Strzelecki}
	\address[D. Strzelecki]{
	\newline\indent Faculty of Mathematics, Informatics and Mechanics,
	\newline\indent University of Warsaw,
	\newline\indent ul. Banacha 2, 02-097 Warsaw, Poland}	
	\email{\href{mailto:dstrzelecki@mimuw.edu.pl}{dstrzelecki@mimuw.edu.pl}}

	\pagestyle{myheadings} \markboth{\underline{F. Bernini, B. Bieganowski, D. Strzelecki}}{
		\underline{TBA}}

\newcommand{\spann}{\mathrm{span}\,}

\newcommand{\triple}[1]{{\left\vert\kern-0.25ex\left\vert\kern-0.25ex\left\vert #1 
    \right\vert\kern-0.25ex\right\vert\kern-0.25ex\right\vert}}

\begin{abstract} 
In the paper, we utilize the recent variational, abstract theorem to show the existence of homoclinic solutions to the Hamiltonian system 
$$
\dot{z} = J D_z H(z, t), \quad t \in \R,
$$
where the Hamiltonian $H : \mathbb{R}^{2N} \times \R \rightarrow \R$ is of the form
$$
H(z, t) = \frac12 Az \cdot z + \Gamma(t) \left( F(z) - \lambda G(z) \right)
$$
for some symmetric matrix $A$. 

\medskip

   
\end{abstract}

\maketitle

\section{Homoclinic solutions to Hamiltonian systems}
In classical Hamiltonian mechanics, the Lagrangian $\cL$ associated with a system is a function that summarizes the dynamics of the entire system, which, in contrast with the Newtonian mechanics that uses forces, considers the energies (kinetic and potential) to generate the equations of motions through the stationary action principle. In this context it is possible to recover the \textit{Euler-Lagrange equations of the motions} for a set of $N$ particles $q_1(t), \dots, q_N(t)$ (with $N$ denoting the space dimension), that is
\[
    \frac{\partial\cL}{\partial q_i} - \frac{d}{dt}\frac{\partial\cL}{\partial\dot{q}_i} = 0.
\]
By computing the Legendre transform of the Lagrangian, we obtain the associated \textit{Hamiltonian equations in $2N$ dimensions}, which, roughly speaking, replaces the velocities used in Newtonian mechanics with the momenta $p_i$ of the particles, and are given by
\begin{equation}\label{H:system}
    \begin{cases}
        \dot{q}=\frac{\partial H}{\partial p}\\
        \dot{p}=-\frac{\partial H}{\partial q},
    \end{cases}
\end{equation}
where $H$ denotes the \textit{Hamiltonian function}.

Despite the names, Hamilton's equations first appear in a paper by Lagrange on perturbation theory, and just later on used by Hamilton as foundations of his analytical mechanics and to give the first exact formulation of the least action principle.

Hamiltonian systems give a good description of those physical phenomena where the energy is (approximately) conserved, from planetary orbits to the motion of particles. Some examples are given by a point of mass $m$ on a line in the presence of a potential, the pendulum, an electric charge $e$ in a given electromagnetic field, as well as the mixing of fluids, and the ray equations describing the trajectories of propagating waves. In particular, Hamiltonian systems find a deep role in the study of celestial mechanics, especially in the $N-$body problem, which is a system of equations that describes the motion of $N-$points masses moving under the influence of their mutual gravitation attraction (we refer to \cite{Arnold,MaWi,MeHa} for a more precise introduction).

From a mathematical point of view, Hamiltonian systems are interesting because of the methods one needs to use. Though the systems admit a variational structure, the associated (action) functional is indefinite (see \cite[Chapter 3]{MaWi}). Therefore, direct methods of the variational approach are hard to apply. For this reason, other theories were (and still are) used in the study of these problems, for instance, topological methods (\cite{GidL,Moser1976,St,Weinstein1973}) and numerical methods - KAM theory (\cite{CaCedL,Sevryuk}).

As already mentioned, the main difficulty in applying variational techniques is given by the strong indefiniteness of the problem: Linking Theorem proved by Rabinowitz (see \cite{Rab78}) was the first milestone allowing a variational approach to this problem. Roughly, this result states that if the problem is settled in a bounded domain and it can be orthogonally split in the direct sum of two subspaces, one of them being finite-dimensional, plus a geometry condition, then the functional associated with the problem has a critical point. In 1979, Benci and Rabinowitz in \cite{BenciRabinowitz79} extended the Linking Theorem allowing a decomposition where both of the subspaces could be of infinite dimension: this result found great applications in the next decades in critical point theory, e.g. see \cite{Benci82,Hofer83,HeinzStuart92,BartschWang97,AbbondandoloMolina00}. At the end of the 90s, Kryszewski and Szulkin extended the Rabinowitz Linking Theorem by considering a new topology (called \textit{weak-strong topology}) that yields the introduction of a new degree of Leray-Schauder type (see \cite{KS}). However, in all the works mentioned until now, it is required that the nonlinearity is nonnegative. Very recently, a further extension was given by the first two authors to sign-changing nonlinearity (\cite{BB}, see also \cite{CW}).

As already remarked above, due to their nature, dynamical systems, and so in particular Hamiltonian systems, are indefinite; therefore, the Linking-type theorems have been widely used to variationally find homoclinic solutions as long as a nonnegative nonlinearity was considered because of the application of Rabinowitz or Kryszewski-Szulkin linking theorem. Our aim in this paper is to provide an existence results via variational method without sign assumptions on the nonlinearity. Let us explain in detail.
                   
Let 
\[
    J := \left[ \begin{array}{cc}
    0 & -I \\
    I & 0
    \end{array} \right] \in \mathbb{M}_{2N \times 2N} (\R)
\]
denote the standard symplectic matrix. We rewrite \eqref{H:system} in the following form 
\begin{equation}\label{eq:hamiltonian}
    \dot{z} = J D_z H (z,t), \quad t \in \R,
\end{equation}
where $z = z(t) \in \R^{2N}$. 
The existence, multiplicity, and properties of solutions to \eqref{eq:hamiltonian} were studied using variational methods (\cite{CZES, SZ, DL, DW, Se1, Se2, HW}) and topological arguments (\cite{B, SW, St}). We emphasize that, to the best of our knowledge, the existence of homoclinic solutions for such systems with sign-changing superquadratic part of the Hamiltonian has not been studied yet. We aim to show the existence of a homoclinic solution to \eqref{eq:hamiltonian} by application of the recent variational approach introduced in \cite{BB} for the Schr\"odinger equation.

Recall that a solution $z : \R \rightarrow \R^{2N}$ is called \textit{homoclinic to 0} if $z \neq 0$ and $z(t) \to 0$ as $|t|\to\infty$. We consider the nonautonomous Hamiltonian $H : \R^{2N} \times \R \rightarrow \R$ of the form
\[
	H(z, t) = \frac12 Az \cdot z + \Gamma(t) \left( F(z) - \lambda G(z) \right), \quad \mbox{with } F(0) = G(0) = 0,
\]
where $F, G : \R^{2N} \rightarrow \R$ are nonlinear functions and $\Gamma : \R \rightarrow (0,+\infty)$. Hence we consider Hamiltonians with a super-quadratic part that is nonautonomous of separated variables, that may be sign-changing. Hereafter we consider a general variable $z$ and we do not split it into \textit{(position, momentum)} pair. The symbols $p$ and $q$ will be used in the different contexts.

Below, $|\cdot |$ denotes the Euclidean norm, $\| \cdot \|_k$ the usual $L^k$-norm and $\lesssim$ the inequality up to a multiplicative constant.

We assume that the matrix $A$ satisfies
\begin{itemize}
\item[(A)] $A \in \mathbb{M}_{2N \times 2N} (\R)$ is a symmetric matrix with $\sigma( JA) \cap i  \R = \emptyset$.
\end{itemize}
We assume also that $F, G : \R^{2N} \rightarrow \R$, $\Gamma : \R \rightarrow (0,+\infty)$ satisfy the following assumptions.
\begin{enumerate}
\item[($\Gamma$)] $\Gamma \in L^\infty (\R)$ is $1$-periodic and $\Ga_0 := \inf_\R \Gamma > 0$.
\item[(F1)] $f := \nabla F : \R^{2N} \to \R^{2N}$ is continuous and odd, and there is $p > 2$ such that
\[
	|f(z)| \lesssim 1+|z|^{p-1} \text{ for all } z \in \R^{2N}.
\]
\item[(F2)] $f(z)=o(|z|)$ as $z \to 0$.
\item[(F3)] There is $q \in (2,p)$ such that $F(z)/|z|^q \to +\infty$ as $|z| \to +\infty$ and $F(z) \geq 0$ for all $z \in \R^{2N}$.
\item[(F4)] The map $\zeta \mapsto \frac{f(\zeta z) \cdot z}{\zeta^{q-1}}$ is nondecreasing on $(0, +\infty)$.
\item[(F5)] There is $\rho > 0$ such that $|f(z) \cdot z| \gtrsim |z|^{p}$ for $|z| \geq \rho$.
\item[(G1)] $g:= \nabla G : \R^{2N} \to \R^{2N}$ is odd, continuous and
\[
	|g(z)| \lesssim 1+|z|^{q-1} \text{ for all } z \in \R^{2N},
\]
where $q$ is given in (F3).
\item[(G2)] $g(z)=o(|z|)$ as $z \to 0$.
\item[(G3)] The map $\zeta \mapsto \frac{g(\zeta z) \cdot z}{\zeta^{q-1}}$ is nonincreasing on $(0,+\infty)$ and there holds
\[
	g(z) \cdot z \geq 0 \text{ for all } z \in \R^{2N}.
\]
\end{enumerate}
These assumptions were considered in the scalar case in \cite{BB}.
In the vectorial case, in addition, we need to consider the following condition.
\begin{enumerate}
\item[(FG)] $|f(z)\cdot w|\gtrsim |g(z)\cdot w||z|^{p-q}$, for $|z|\geq \rho$ and $w\in\R^{2N}$, where $\rho$ is given in (F5).
\end{enumerate}
In particular one can consider $F(z) = \frac1p |z|^p$ and $G(z) = \frac1q |z|^q$, where $2 < q < p$. For more examples we refer the reader to \cite{BB}.

We aim to prove the following result.
\begin{Th}\label{existence:theorem}
Suppose that (A), ($\Ga$), (F1)--(F5), (G1)--(G3), (FG) hold. Then, if $\rho > 0$ and $\lambda > 0$ are sufficiently small, there exists a nontrivial, homoclinic solution $z \in H^{1} (\R; \R^{2N})$ of \eqref{eq:hamiltonian}.
\end{Th}

The paper is organized as follows. In Section \ref{sect:functional} we introduce the functional setting and the variational formulation of the problem. Section \ref{sect:abstract} is devoted to the abstract critical point theory, it introduces the topology on the energy space based on \cite{KS}, and the abstract theorem from \cite{BB}. In Section \ref{sect:4} we verify all the assumptions of the abstract theorem and therefore we find a Cerami sequence that is not vanishing - see \eqref{Cer}. Then, in Section \ref{sect:5} we verify the boundedness of such a sequence, which allows us to conclude the proof of the existence of homoclinic solutions in Section \ref{sect:6} using concentration-compactness principles.

\section{Functional setting}\label{sect:functional}

Consider the fractional Sobolev space $X := H^{1/2} (\R; \R^{2N}) := \left( H^{1/2} (\R; \R) \right)^{2N}$, where
$$
H^{1/2} (\R; \R) := \left\{ w \in L^2 (\R) \ : \ \iint_{\R \times \R} \frac{|w(t)-w(s)|^2}{|t-s|^2 } \, dt \, ds < +\infty \right\}
$$
and define the functional $\cJ : X \rightarrow \R$ by
\begin{equation}
\label{energy:functional}
	\cJ(z) := \frac12 \int_{\R} (-J \dot{z} - Az ) \cdot z \, dt - \int_{\R} \Ga(t) F(z) \, dt + \lambda \int_{\R} \Ga(t) G(z) \, dt.
\end{equation}
Observe that $X = \cD( |\mathbf{A}|^{1/2} )$ is the domain of $|\mathbf{A}|^{1/2}$, where $\mathbf{A} z := -J \dot{z} - Az$. Then, under assumption (A) - see \cite{Arioli, Stuart, SZ}, there is an orthogonal splitting $X = X^+ \oplus X^-$ such that the quadratic form
\[
	z \mapsto \int_{\R} (-J \dot{z} - Az ) \cdot z \, dt
\]
is positive definite on $X^+$ and negative definite on $X^-$. Hence we set
\[
	\| z^+ \|^2 := \int_{\R} (-J \dot{z}^+ - Az^+ ) \cdot z^+ \, dt
\]
and
\[
	\| z^- \|^2 := -\int_{\R} (-J \dot{z}^- - Az^- ) \cdot z^- \, dt.
\]
Thus, we introduce the norm on $X$ given by
\[
	\|z\|^2 := \|z^+\|^2 + \|z^-\|^2, \quad z = z^+ + z^- \in X.
\]
Then, under (F1), (F2), (G1), (G2), $\cJ$ is of $\cC^1$ class on $X$ and its critical points are weak solutions to \eqref{eq:hamiltonian}. Namely, $z \in X$ is a weak solution if
\begin{equation}\label{eq:weakSol}
0 = \cJ'(z)(v) = \int_{\R} (-J \dot{z} - Az) \cdot v \, dt - \int_\R \Ga(t) f(z) \cdot v \, dt + \lambda \int_\R \Ga(t) g(z) \cdot v \, dt
\end{equation}
for every $v \in X$. However, note that if $z \in X$, then $z \in L^s (\R; \R^{2N})$ for every $s \in [2, \infty)$ due to Sobolev embeddings. Then, in view of ($\Ga$), (F1) and (F2), we note that $\Gamma(\cdot) f(z(\cdot)) \in L^2 (\R; \R^{2N})$. Similarly $\Gamma(\cdot) g(z(\cdot)) \in L^2 (\R; \R^{2N})$. Hence, if $z \in X$ is a critical point of $\cJ$, \eqref{eq:weakSol} leads to $\dot{z} \in L^2(\R; \R^{2N})$ and therefore $z \in H^1 (\R; \R^{2N})$. Since $H^1 (\R; \R^{2N}) \subset \cC (\R; \R^{2N})$, $z(t) \to 0$ as $|t| \to \infty$. Thus every critical point of $\cJ$ is a homoclinic solution to \eqref{eq:hamiltonian}, cf. \cite[Lemma 10.6]{Stuart}, \cite[Proposition 3.2]{Arioli}.

We recall that the projections $X \to X^{\pm}$ are continuous in $L^q(\R; \R^{2N})$ for $q > 2$, in the sense that there is $\kappa \geq 1$ such that 
\begin{equation}
\label{projection:inequality}
	\|z^{\pm}\|_q \leq \kappa \|z\|_q
\end{equation}
for $z \in X$ (see \cite[Proposition 7]{troestler2012bifurcation}). Moreover, by (A), there exists a constant $\mu_0>0$ such that $\| \cdot \|$ satisfies
\begin{equation}
	\label{mu0:inequality}
		\mu_0\|z\|_2 \leq \|z\|, \quad z \in X.
\end{equation}

\section{\texorpdfstring{$\tau$}{tau}-topology and critical point theory}\label{sect:abstract}

Let $(X, \| \cdot \|)$ be a real, separable Hilbert space with the inner product denoted by $\langle \cdot, \cdot \rangle$. Assume that there is an orthogonal splitting $X = X^+ \oplus X^-$. Then every $z \in X$ has a unique decomposition $z = z^+ + z^-$, where summands satisfy $z^\pm \in X^\pm$. We introduce a new topology $\tau$ in the space $X$, see \cite{KS}. Let $(e_k)_{k=1}^\infty \subset X^-$ be a complete orthonormal sequence in the space $X^-$. Then we define a norm $\triple{\cdot}$ in $X$ by
\[
	\triple{z} := \max \left\{ \| z^+ \|, \sum_{k=1}^\infty \frac{1}{2^{k+1}} \left| \langle z^-, e_k \rangle \right| \right\}.
\]
Let $\tau$ denote the topology on $X$ generated by $\triple{\cdot}$. We note that $\tau$ is weaker than the topology generated by the norm $\| \cdot \|$ and that the following inequalities hold
\[
	\|z^+\| \leq \triple{z} \leq \|z\|.
\]
We also recall that for bounded sequences $(z_n) \subset X$ the following equivalence holds true (see e.g. \cite[Remark 2.1(iii)]{KS}) 
\[
z_n \stackrel{\tau}{\to} z \quad \Longleftrightarrow \quad z_n^+ \to z^+ \mbox{ and } z_n^- \weakto z^-. 
\]

Let $\cJ : X \rightarrow \R$ be a nonlinear functional. For $z \in X \setminus X^-$ and $R > r > 0$ we introduce the following sets:
\begin{align*}
\cS^+_r &:= \{ z^+ \in X^+ \ : \ \|z^+\| = r \} \\
M(z) &:= \{ tz + v^- \ : \ v^- \in X^-, \ t \geq 0, \, \|tz + v^- \| \leq R \}.
\end{align*}
It is clear that $M(z) \subset \R_+ z^+ \oplus X^-$, where $\R_+ := [0,\infty)$, is a manifold with the boundary
\[
	\partial M(z) = \{ v^- \in X^- \ : \ \| v^- \| \leq R \} \cup \{tz + v^- \ : \ v^- \in X^-, \ t > 0, \ \|tz + v^- \|=R \}.
\]
We are working under the following assumptions.
\begin{itemize}
\item[(A1)] $\cJ$ is of $\cC^1$-class and $\cJ(0) = 0$.
\item[(A2)] $\cJ'$ is sequentially weak-to-weak* continuous.
\end{itemize}
Let $\cP \subset X \setminus X^-$ be a nonempty set. We assume that
\begin{itemize}
\item[(A3)] there are $\delta > 0$ and $r > 0$ such that for every $z \in \cP$ there is radius $R=R(z) > r$ with
\[
	\inf_{\cS_r^+} \cJ > \max \left\{ \sup_{\partial M(z)} \cJ, \sup_{\triple{v} \leq \delta} \cJ(v) \right\}.
\]
\end{itemize}
Let $A \subset X$ and $I :=[0,1]$. Let $h : A \times I \rightarrow X$. We consider the following conditions. 
\begin{itemize}
\item[(h1)] $h$ is $\tau$-continuous, i.e. $h(v_n, t_n) \stackrel{\tau}{\to} h(v, t)$ for $v_n \stackrel{\tau}{\to} v$ and $t_n \to t$;
\item[(h2)] $h(z,0)=z$ for $z \in A$;
\item[(h3)] $\cJ(z)\geq \cJ(h(z,t))$ for $(z,t) \in A \times I$;
\item[(h4)] for every $(z,t) \in A \times I$ there is an open - in the product topology of $(X, \tau)$ and $(I, |\cdot |)$ - neighborhood $W \subset X \times I$ of $(z,t)$ such that $\{v - h(v,s) \ : \ (v,s) \in W \cap (A \times I)\}$ is contained in a finite-dimensional subspace of $X$.
\end{itemize}

\begin{Th}[{\cite[Theorem 2.1]{BB}}]\label{th:abstrExistence}
Suppose that $\cJ$ satisfy (A1)--(A3). Then there is a  Cerami sequence $(z_n) \subset X$ bounded away from zero, i.e. a sequence such that
\begin{equation}\label{Cer}
    \sup_n \cJ(z_n) \leq c, \quad (1+\|z_n\|) \cJ'(z_n) \to 0 \mbox{ in } X^*, \quad \inf_n \triple{z_n} \geq \frac{\delta}{2},
\end{equation}
where
\[
	c := \inf_{z \in \cP} \inf_{h \in \Gamma(z)} \sup_{z' \in M(z)} \cJ(h(z',1)) \geq \inf_{\cS_r^+} \cJ > 0
\]
and
\[
	\Gamma(z) := \left\{ h \in \cC(M(z) \times [0,1]) \ : \ h \mbox{ satisfies (h1)--(h4)} \right\} \neq \emptyset.
\]
 \end{Th}

\section{Existence of a Cerami sequence - verification of (A1)--(A3)}\label{sect:4}

We will show that, under the assumptions of Theorem \ref{existence:theorem}, there is a Cerami sequence for $\cJ$ bounded away from zero. Hence, we will check conditions (A1)--(A3) of Theorem \ref{th:abstrExistence} with
\[
	\cP := X^+ \setminus \{0\}.
\]
It is classical to check the following useful estimates for $F$ and $G$.
\begin{Prop}
	\label{f,g:growth}
	Suppose (F1)-(F2) and (G1)-(G2) hold, then for every $\eps > 0$ there exists $C_{f,\eps}>0$ and $C_{g,\eps} > 0$ such that
	\begin{align}
		\label{f:growth}
		|f(z)| &\leq \eps|z| + C_{f,\eps}|z|^{p-1},\\
		\label{g:growth}
		|g(z)| &\leq \eps|z| + C_{g,\eps}|z|^{q-1}.
	\end{align}
\end{Prop}

\begin{Rem}
Note that, by a simple integration, we obtain the following estimates:
\begin{equation}
	\label{F:growth}
	F(z) \leq \eps|z|^2 + C_{F,\eps}|z|^p
\end{equation}
and
\begin{equation}
	\label{G:growth}
	G(z) \leq \eps|z|^2 + C_{G,\eps}|z|^q.
\end{equation}
\end{Rem}
\noindent Repeating the proof of \cite[Lemma 3.2]{BB} we note the following property. 

\begin{Lem}
Suppose that $(F3)$ holds. For every $\eps>0$ there exists $C_{\eps}>0$ such that
\begin{equation}
	\label{F:bound:below}
		F(z) \geq C_{\eps}|z|^{q} - \eps|z|^2,
\end{equation}
for every $z \in \R^{2N}$, where $2 < q < +\infty$ is the exponent given in $(F3)$. Moreover, we can assume that
\begin{equation}\label{C:ineq}
    C_\varepsilon \leq C_{G,\varepsilon}.
\end{equation}
\end{Lem}

The following lemma states that under our hypotheses, both functions $f$ and $g$ satisfy Ambrosetti-Rabinowitz-type conditions.
\begin{Lem}
Suppose that $(F3)$ and $(F4)$ hold. Then
\begin{equation}
	\label{AR:f}
		0 \leq qF(z) \leq f(z)\cdot z. 
\end{equation}
Suppose that $(G3)$ holds. Then
\begin{equation}
	\label{AR:g}
	0 \leq g(z)\cdot z \leq qG(z). 
\end{equation}
\end{Lem}
\begin{proof}
	By (F4), we get that
 $$
 q F(z) = q \int_0^1 f(sz) \cdot z \, ds = q \int_0^1 \frac{f(sz) \cdot z}{s^{q-1}} s^{q-1} \, ds \leq q f(z) \cdot z \int_0^1 s^{q-1} \, dt = f(z) \cdot z 
 $$
and taking (F3) into account, we obtain \eqref{AR:f}. A similar argument leads to \eqref{AR:g}.
\end{proof}

Rewrite the energy functional \eqref{energy:functional} in the form
$$
	\cJ(z) 
	 = \frac12\|z^+\|^2-\frac12\|z^-\|^2 - \int_{\R} \Ga(t) F(z) \, dt + \lambda \int_{\R} \Ga(t) G(z) \, dt.
$$
It is classical to check assumptions (A1) and (A2) of Theorem \ref{th:abstrExistence}, and it is enough to verify (A3) for sufficiently small $\lambda > 0$. We divide the proof into three steps.

\bigskip

\textbf{Step 1.} \textit{There is $r > 0$ such that $\inf_{\cS_r^+} \cJ > 0$.}\\
Fix $z^+ \in X^+$. Using \eqref{F:growth} and the continuous Sobolev embedding $H^{\frac12}(\R,\R^{2N}) \hookrightarrow L^p(\R,\R^{2N})$, we have
\begin{align*}
	\cJ(z^+) 
	&\geq \frac12\|z^+\|^2 - \int_\R \Ga(t) F(z^+) \, dt \geq \frac12\|z^+\|^2 - \eps \|\Gamma\|_\infty \|z^+\|_2^2 - C_{F,\eps} \|\Gamma\|_\infty \|z^+\|^p_p \\
	&\geq \frac12\|z^+\|^2 - \eps C \|z^+\|^2 - \widetilde{C}_\varepsilon \|z^+\|^p =\left(\frac12-\eps  C\right)\|z^+\|^2 - \widetilde{C}_\varepsilon \|z^+\|^p \\
	&= \|z^+\|^2\left(\frac12 - \eps C -\widetilde{C}_\varepsilon \|z^+\|^{p-2}\right) = \|z^+\|^2\left(\frac12 - \eps C- \widetilde{C}_\varepsilon r^{p-2}\right),
\end{align*}
where we considered $\eps < \frac{1}{2 C}$ and $\|z^+\| = r$. Now, choosing $r>0$ such that 
\[
    r < \left[\left(\frac12 -\eps C\right)\widetilde{C}_\varepsilon^{-1} \right]^{\frac{1}{p-2}} 
\]    
we get
\[
	\inf_{\|z^+\| = r} \cJ(z^+) > 0.
\]

\textbf{Step 2.} \textit{For $z \in \cP$, there is radius $R(z) > r$ such that $\sup_{\partial M(z)} \cJ \leq 0$.}\\
Fix $z \in \cP = X^+ \setminus \{0\} \subset X \setminus X^-$ and let $z_n \in \R^+z \oplus X^-$, namely, $z_n=t_nz + z_n^-$ for some $t_n \geq 0$, $z_n^- \in X^-$, and we can assume $\|z\|=1$. Using \eqref{G:growth}, \eqref{F:bound:below} and \eqref{mu0:inequality}, we get
\begin{align*}
	\cJ(z_n)
	&=\cJ(t_nz+z_n^-) \\
	&=\frac12t_n^2-\frac12\|z_n^-\|^2 - \int_{\R} \Ga(t) F(z_n) \, dt + \lambda \int_{\R} \Ga(t) G(z_n) \, dt \\
	&\leq \frac12t_n^2-\frac12\|z_n^-\|^2 + \eps ( \Ga_0 +\lambda \|\Ga\|_\infty)\int_\R |z_n|^2 \, dt - C_{\eps}\Ga_0 \int_\R |z_n|^q \, dt + \lambda C_{G,\eps} \|\Ga\|_\infty \int_\R |z_n|^q \, dt \\
	& \leq \frac12t_n^2-\frac12\|z_n^-\|^2 + \eps ( \Ga_0 +\lambda \|\Ga\|_\infty )\frac{1}{\mu_0^2}\|z_n\|^2  - C_\eps\Gamma_0 \|z_n\|^q_q \\
    &\qquad + \lambda2^{q-1}C_{G,\eps}\|\Gamma\|_\infty\left(t_n^q \|z\|_q^q + \|z_n^-\|_q^q  \right) \\
	& \leq \left(\frac12 +\frac{\eps(\Gamma_0+\lambda\|\Gamma\|_\infty)}{\mu_0^2}\right)t_n^2 + \left(-\frac12 +\frac{\eps(\Gamma_0+\lambda\|\Gamma\|_\infty)}{\mu_0^2}\right)\|z_n^-\|^2 \\
    &\qquad - \frac{C_\eps\Gamma_0}{2\kappa^q}\left( t_n^q \|z\|^q_q + \|z_n^-\|^q_q\right) + \lambda2^{q-1}C_{G,\eps}\|\Gamma\|_\infty\left(t_n^q \|z\|_q^q + \|z_n^-\|_q^q  \right),
\end{align*}
where the last inequality follows from \eqref{projection:inequality}: indeed, for $z \in X$, $\|z^\pm\|_q \leq \kappa \|z\|_q$, hence
\[
	\|z^+\|_q^q + \|z^-\|_q^q \leq 2\kappa^q \|z\|_q^q.
\]
Assuming that $\lambda 
        < \frac{\Ga_0}{\|\Ga\|_\infty}\frac{C_\eps}{C_{G,\eps}}\frac{1}{(2\kappa)^q}$ and choosing $\eps=\frac{\mu_0^2}{4(\Ga_0 +\|\Ga\|_\infty)}$, we get
\begin{align*}
	\cJ(z_n)
	&=\cJ(t_nz+z_n^-) \\
	&\leq \left(\frac12 +\frac{\eps(\Gamma_0+\lambda\|\Gamma\|_\infty)}{\mu_0^2}\right)t_n^2 + \left(-\frac12 +\frac{\eps(\Gamma_0+\lambda\|\Gamma\|_\infty)}{\mu_0^2}\right)\|z_n^-\|^2 \\
    & \quad + \left(\lambda2^{q-1}C_{G,\eps}\|\Gamma\|_\infty - \frac{C_\eps\Gamma_0}{2\kappa^q}\right)t_n^q\|z\|^q_q + \left(\lambda2^{q-1}C_{G,\eps}\|\Gamma\|_\infty - \frac{C_\eps\Gamma_0}{2\kappa^q}\right)\|z_n^-\|^q_q \\
    &\leq \frac34 t_n^2 -\frac14 \|z_n^-\|^2 \\
    & \quad + \left(\lambda2^{q-1}C_{G,\eps}\|\Gamma\|_\infty - \frac{C_\eps\Gamma_0}{2\kappa^q}\right)t_n^q\|z\|^q_q + \left(\lambda2^{q-1}C_{G,\eps}\|\Gamma\|_\infty - \frac{C_\eps\Gamma_0}{2\kappa^q}\right)\|z_n^-\|^q_q.
\end{align*}
Eventually, we obtain that
\[
	\cJ(z_n) \to -\infty \quad \text{ as } \|t_nz+z_n^-\| \to +\infty.
\]
In particular, if $t_n=0$  then $\cJ(z_n^-) \leq 0$, namely, $\sup_{X^-} \cJ \leq 0$. Therefore, for a sufficiently large $R(z)$, we get $\sup_{\partial M(z)}\cJ(z) \leq 0$.	

\begin{Rem}
    Observe that, using \eqref{C:ineq},
    \[
        \lambda2^{q-1}C_{G,\eps}\|\Gamma\|_\infty - \frac{C_\eps\Gamma_0}{2\kappa^q} < 0 
    \]
    if and only if 
    \[
        \lambda 
        < \frac{\Ga_0}{\|\Ga\|_\infty}\frac{C_\eps}{C_{G,\eps}}\frac{1}{(2\kappa)^q}.
    \]
    Then also $\lambda \leq \frac{\Ga_0}{\|\Ga\|_\infty}\frac{1}{(2\kappa)^q} 
        \leq 
        \frac{\Ga_0}{\|\Ga\|_\infty} \leq 1.$
\end{Rem}
{Step 3.} \textit{There is $\delta > 0$ such that $\sup_{\triple{z} \leq \delta} \cJ(z) < \inf_{\cS_r^+} \cJ$.} \\
Let $z \in X$. By \eqref{G:growth}, \eqref{F:bound:below}, \eqref{C:ineq} and \eqref{mu0:inequality}
\begin{align*}
	\cJ(z)
	&=\frac12\|z^+\|^2-\frac12\|z^-\|^2 - \int_{\R} \Gamma(t)F(z) \, dt + \lambda \int_{\R} \Gamma(t)G(z) \, dt \\
	&\leq \frac12\|z^+\|^2-\frac12\|z^-\|^2 - \Gamma_0C_\eps\|z\|^q_q + \Gamma_0\eps\|z\|^2_2 + \lambda\|\Gamma\|_\infty\eps\|z\|^2_2 + \lambda \|\Gamma\|_\infty C_{G,\eps}\|z\|^q_q\\
    &\leq \frac12\|z^+\|^2-\frac12\|z^-\|^2 + \underbrace{\frac{\eps}{\mu_0^2}(\Gamma_0+\lambda\|\Gamma\|_\infty)}_{=: D = D(\varepsilon, \lambda)}\|z\|^2 - \Gamma_0 C_\eps\|z\|^q_q + \lambda\|\Gamma\|_\infty C_{G,\eps}\|z\|^q_q\\
	&\leq \frac12\|z^+\|^2-\frac12\|z^-\|^2 + D \left( \|z^+\|^2 + \|z^-\|^2 \right) -\Gamma_0 C_\eps\|z\|^q_q + \lambda\|\Gamma\|_\infty C_{G,\eps}\|z\|^q_q\\
	&\leq \left(\frac12 + D\right)\|z^+\|^2 - \left(\frac12 - D\right)\|z^-\|^2 + \left(\lambda\|\Gamma\|_\infty C_{G,\eps} - \Gamma_0 C_\eps\right)\|z\|^q_q \\
	&\leq \left(\frac12 + D \right)\|z^+\|^2 \leq \left(\frac12 + D \right) \triple{z} \to 0 \text{ as } \triple{z} \to 0,
\end{align*}
where we chose $\lambda$ and $\eps$ as in Step 2 (in particular, $\frac12-D > 0$).

\section{Boundedness of Cerami-type sequences}\label{sect:5}

Now, we are going to discuss the boundedness of a Cerami sequence for $\cJ$, whose existence was established in the previous Section.

\begin{Rem}\label{Rem:monotonicity}
Observe that (F4), (G3) and the oddness of $f,g$ imply the following property:
\[
	\left| \frac{g(z) \cdot z}{f(z) \cdot z} \right| \leq \sup_{|w|= \rho} \frac{g(w) \cdot w}{f(w) \cdot w}
\]
for $|z| \geq \rho$. (F5) ensures that $\{ |z| \geq \rho \} \ni z \mapsto \frac{g(z) \cdot z}{f(z) \cdot z} $ is well-defined.
\end{Rem}

\begin{Lem}\label{lem:cerami_bdd}
Suppose that $\lambda > 0$ and $\rho > 0$ in (F5) are sufficiently small. Let $(z_n) \subset X$ satisfies
\[
	\cJ(z_n) \leq \beta, \quad (1+\|z_n\|) \cJ'(z_n) \to 0
\]
for some $\beta \in \R$. Then $(z_n)$ is bounded in $X$. In particular, any Cerami sequence for $\cJ$ is bounded.
\end{Lem}

\begin{proof}
Take $(z_n) \subset X$ as in the statement. Note that
\[
	\|z_n\|^2 = \|z_n^+\|^2 + \|z_n^-\|^2 = \int_{\R} \Ga(t) (f(z_n) - \lambda g(z_n)) \cdot (z_n^+ - z_n^-) \, dt + o(1),
\]
and
\begin{align*}
	&\quad \int_{\R} \Ga(t) (f(z_n) - \lambda g(z_n)) \cdot (z_n^+ - z_n^-) \, dt\\
	&= \underbrace{\int_{|z_n| < \rho} \Ga(t) (f(z_n) - \lambda g(z_n)) \cdot (z_n^+ - z_n^-) \, dt}_{=:I_1} + \underbrace{\int_{|z_n| \geq \rho} \Ga(t) (f(z_n) - \lambda g(z_n)) \cdot (z_n^+ - z_n^-) \, dt}_{=:I_2}.
\end{align*}
To estimate $I_1$ we fix $\eps > 0$, and we obtain, using $(\Gamma)$, \eqref{f:growth}, and \eqref{g:growth},
\begin{align*}
	I_1 
	&\leq \|\Ga\|_\infty \int_{|z_n| < \rho} |f(z_n)-\lambda g(z_n)| \left| z_n^+ - z_n^- \right| \, dt  \\
	&\leq \eps (1+\lambda) \|\Ga\|_\infty \int_{|z_n| < \rho} |z_n| \left| z_n^+ - z_n^- \right| \, dt + C_{f,\eps} \|\Ga\|_\infty \int_{|z_n|<\rho} |z_n|^{p-1} \left| z_n^+ - z_n^- \right| \, dt \\
	&\quad  + C_{g,\eps} \lambda \|\Ga\|_\infty \int_{|z_n|<\rho} |z_n|^{q-1} \left| z_n^+ - z_n^- \right| \, dt \\
	&\leq \left( \eps(1+\lambda) + C_{f,\eps} \rho^{p-2} + \lambda C_{g,\eps} \rho^{q-2} \right)  \|\Ga\|_\infty \int_{|z_n| < \rho} |z_n| \left| z_n^+ - z_n^- \right| \, dt \\
    &\leq \left( \eps(1+\lambda) + C_{f,\eps} \rho^{p-2} + \lambda C_{g,\eps} \rho^{q-2} \right) \|\Ga\|_\infty \left(\|z_n^+\|^2_2 + \|z_n^-\|^2_2\right) \\
    &\leq \frac{2\kappa}{\mu_0^2}\left( \eps(1+\lambda) + C_{f,\eps} \rho^{p-2} + \lambda C_{g,\eps} \rho^{q-2} \right) \|\Ga\|_\infty \|z_n\|^2,
\end{align*}
where we used \eqref{projection:inequality} and \eqref{mu0:inequality} in the last inequality.

To estimate $I_2$, we note the inequality, valid for $z = z^+ + z^- \in X$, at points where $|z| \geq \rho$, thanks to (FG)
\begin{align}
	\label{I2-1} | f(z) \cdot (z^+ - z^-) | \gtrsim |g(z)\cdot(z^+ - z^-)| |z|^{p-q},
\end{align}
Hence,
\begin{align*}
	I_2 
	&= \int_{|z_n| \geq \rho} \Ga(t) f(z_n) \cdot (z_n^+ - z_n^-) \left( 1 - \lambda \frac{g(z_n) \cdot (z_n^+ - z_n^-)}{f(z_n) \cdot (z_n^+ - z_n^-)} \right) \, dt \\
	&\leq \|\Ga\|_\infty \int_{|z_n| \geq \rho} |f(z_n)| |z_n^+ - z_n^-|  \left( 1 + \lambda \frac{|g(z_n) \cdot (z_n^+ - z_n^-)|}{|f(z_n) \cdot (z_n^+ - z_n^-)|} \right)  \, dt.
\end{align*}
Now, using (F1), (F5), \eqref{I2-1}, H\"older inequality, and \eqref{projection:inequality}, we get
\begin{align*}
	I_2 \lesssim \|\Ga\|_\infty \int_{|z_n| \geq \rho} |z_n|^{p-1} |z_n^+ - z_n^-|  \left( 1 + \frac{\lambda}{|z_n|^{p-q}} \right)  \, dt \leq 2\kappa\|\Ga\|_\infty\left( 1 + \frac{\lambda}{\rho^{p-q}} \right) \int_{\R} |z_n|^p  \, dt.
\end{align*}
To estimate the $L^p$-norm of $z_n$ we observe the following
\begin{align*}
	\beta + o(1) &\geq \cJ(z_n) - \frac12 \cJ'(z_n)(z_n) = \int_{\R} \Ga(t) \Phi(z_n) \, dt,
\end{align*}
where we set
$$
    \Phi(z) := \frac12 f(z)\cdot z - F(z) + \lambda G(z) - \frac{\lambda}{2} g(z) \cdot z
$$
for a simplicity of notation. Using $(\Gamma)$, (F5), \eqref{AR:f}, \eqref{AR:g}, Remark \ref{Rem:monotonicity} and choosing $\lambda$ so small that $1-\lambda  \sup_{|w|= \rho} \frac{g(w) \cdot w}{f(w) \cdot w} > 0$, we get
\begin{align*}
    &\beta + o(1) + \|\Ga\|_\infty \int_{|z_n| < \rho} |\Phi(z_n)| \, dt \geq \beta + o(1) - \int_{|z_n| < \rho} \Ga(t) \Phi(z_n) \, dt \\
    &\qquad= \beta + o(1) - \int_{\R} \Ga(t) \Phi(z_n) \, dt + \int_{|z_n| \geq \rho} \Ga(t) \Phi(z_n) \, dt \geq \int_{|z_n| \geq \rho} \Ga(t) \Phi(z_n) \, dt \\
    &\qquad= \int_{|z_n| \geq \rho} \Ga(t) \left[ \frac12 f(z_n)\cdot z_n - F(z_n) + \lambda G(z_n) - \frac{\lambda}{2} g(z_n)\cdot z_n \right] \, dt \\
    &\qquad\geq \left( \frac12 - \frac1q \right) \int_{|z_n| \geq \rho} \Ga(t) \left( f(z_n) \cdot z_n - \lambda g(z_n) \cdot z_n \right) \, dt \\
    &\qquad= \left( \frac12 - \frac1q \right) \int_{|z_n| \geq \rho} \Ga(t) \left(1 - \lambda \frac{g(z_n) \cdot z_n}{f(z_n) \cdot z_n} \right) f(z_n) \cdot  z_n \, dt \\
    &\qquad\geq \left( \frac12 - \frac1q \right) \left(1 - \lambda \sup_{|w|= \rho} \frac{g(w) \cdot w}{f(w) \cdot w} \right) \int_{|z_n| \geq \rho} \Ga(t)  f(z_n) \cdot z_n \, dt \\
    &\qquad\gtrsim \underbrace{\left(1 - \lambda  \sup_{|w|= \rho} \frac{g(w) \cdot w}{f(w) \cdot w} \right)}_{=:E(\lambda,\rho)} \int_{|z_n| \geq \rho} |z_n|^p \, dt.
\end{align*}
Thus
\begin{equation}\label{1}
    \int_{|z_n| \geq \rho} |z_n|^p \, dt \leq C \left[ E(\lambda,\rho)
    \right]^{-1} 
    \left( \beta + \int_{|z_n| < \rho} \Gamma(t)|\Phi(z_n)| \, dt \right) + o(1)
\end{equation}
for some constant $C > 0$ independent of $n$, $\lambda$ and $\rho$.\\ 
Moreover,
\begin{equation}
    \label{2}
    \int_{|z_n| < \rho}\Ga(t)|\Phi(z_n)| \, dt \leq \|\Ga\|_\infty\int_{|z_n| < \rho}\frac{|\Phi(z_n)|}{|z_n|^2}|z_n|^2 \, dt \leq \|\Ga\|_\infty\sup_{|v|<\rho}\frac{|\Phi(v)|}{|v|^2}\int_{\R}|z_n|^2 \, dt.
\end{equation}
Therefore, using \eqref{mu0:inequality}, \eqref{1}  and, \eqref{2}, we get
{\allowdisplaybreaks
\begin{align*}
    I_2 
    &\lesssim \underbrace{\left( 1 + \frac{\lambda}{\rho^{p-q}} \right) 2\kappa\|\Ga\|_{\infty}}_{=:D=D(\lambda, \rho)}  \left( \int_{|z_n| < \rho} |z_n|^p  \, dt + \int_{|z_n| \geq \rho} |z_n|^p  \, dt \right) \\
    &\leq D\left( \int_{|z_n| < \rho} |z_n|^p  \, dt + C\left[E(\lambda,\rho)
    \right]^{-1} \left( \beta + \int_{|z_n| < \rho} \Gamma(t)|\Phi(z_n)| \, dt \right) \right) + o(1) \\
    &\leq D \left( \rho^{p-2} \|z_n\|_2^2 + C\left[E(\lambda,\rho)
    \right]^{-1} \left( \beta + \int_{|z_n| < \rho} \Gamma(t)|\Phi(z_n)| \, dt \right) \right)+ o(1) \\
    &\leq D \left( \frac{\rho^{p-2}}{\mu_0^2} \|z_n\|^2 + \frac{C}{E(\lambda,\rho)
    } \left( \beta + \|\Ga\|_{\infty}\sup_{|v|<\rho}\frac{\Phi(v)}{|v|^2} \int_{|z_n| < \rho} |z_n|^2 \, dt \right) \right)+ o(1) \\
    &\leq D \left( \frac{\rho^{p-2}}{\mu_0^2} \|z_n\|^2 + \frac{C}{E(\lambda,\rho)
    } \left( \beta + \|\Ga\|_{\infty}\frac{1}{\mu_0^2}\sup_{|v|<\rho}\frac{\Phi(v)}{|v|^2} \|z_n\|^2 \right) \right)+ o(1) \\
    &= D \left( \frac{\rho^{p-2}}{\mu_0^2} + \frac{C\|\Ga\|_{\infty}\sup_{|v|<\rho}\frac{\Phi(v)}{|v|^2} }{\mu_0^2 E(\lambda,\rho)
    } \right) \|z_n\|^2 + \underbrace{\frac{CD\beta}{E(\lambda,\rho)
    }}_{=:\widetilde{D}=\widetilde{D}(\lambda,\rho)} + o(1) \\
    &\leq D \left( \frac{\rho^{p-2}}{\mu_0^2} + \frac{C \|\Ga\|_{\infty}\sup_{|v|<\rho}\frac{\Phi(v)}{|v|^2} }{\mu_0^2 E(\lambda,\rho)
    } \right) \|z_n\|^2 + \widetilde{D} + o(1). \\
\end{align*}}
Finally,
\begin{align*}
    \|z_n\|^2 &= I_1 + I_2 + o(1)\\
    & \leq \gamma \left[ \frac{2\kappa}{\mu_0^2}\left( \eps(1+\lambda) + C_{f,\eps} \rho^{p-2} + \lambda C_{g,\eps} \rho^{q-2} \right) \|\Ga\|_\infty \|z_n\|^2 \right. \\
    & \left. \quad +\frac{1}{\mu_0^2} D(\lambda,\rho) \left( \rho^{p-2} + \frac{C \|\Ga\|_{\infty}\sup_{|v|<\rho}\frac{\Phi(v)}{|v|^2} }{E(\lambda,\rho)
    } \right) \|z_n\|^2 + \widetilde{D} \right] + o(1) \\
    & = \gamma \frac{2\kappa \|\Ga\|_\infty }{\mu_0^2}\Lambda \|z_n\|^2 + \gamma \widetilde{D} + o(1),
\end{align*}
for some $\gamma > 0$, where 
\begin{align*}
    \Lambda :&=  \eps(1+\lambda) + C_{f,\eps} \rho^{p-2} + \lambda C_{g,\eps} \rho^{q-2} \\
    & \quad + \left( 1 + \frac{\lambda}{\rho^{p-q}} \right) \left( \rho^{p-2} + \frac{C \|\Ga\|_{\infty}\sup_{|v|<\rho}\frac{\Phi(v)}{|v|^2} }{E(\lambda,\rho)
    } \right)
\end{align*}
Hence, the proof is completed if $\Lambda < \frac{ \mu_0^2}{2 \gamma \kappa  \|\Ga\|_\infty } =: \Lambda_0$ for sufficiently small $\rho > 0$ and $\lambda > 0$.

Remembering that, in general, since $\lambda \leq 1$, we have the following estimate
\begin{align*}
    \Lambda \leq 2\varepsilon + (C_{f,\eps} + 1) \rho^{p-2} + (C_{g,\eps} + 1) \rho^{q-2} + \left( 1 + \frac{\lambda}{\rho^{p-q}} \right) \frac{C \|\Ga\|_{\infty}\sup_{|v|<\rho}\frac{|\Phi(v)|}{|v|^2} }{E(\lambda,\rho)}.
\end{align*}
Take $\eps = \frac{\Lambda_0}{16}$. Choose $\rho > 0$ so small that
\begin{align*}
&2 C \|\Ga\|_\infty \sup_{|v|<\rho}\frac{|\Phi(v)|}{|v|^2} \leq \frac{\Lambda_0}{16}, \\
&(C_{f,\eps} + 1) \rho^{p-2} \leq \frac{\Lambda_0}{8}, \quad (C_{g,\eps} + 1) \rho^{q-2} \leq \frac{\Lambda_0}{8}.
\end{align*}
Note that it is possible, because $\lim_{|v| \to 0} \frac{|\Phi(v)|}{|v|^2} = 0$. Then, for such $\rho$ choose $\lambda$ so small that
$$
E(\lambda, \rho) = 1 - \lambda\sup_{|w|= \rho} \frac{g(w) \cdot w}{f(w) \cdot w} \geq \frac12, \quad \frac{\lambda}{\rho^{p-q}} \leq 1.
$$
Then observe that
\begin{align*}
    \Lambda &\leq \frac{3 \Lambda_0}{8} + 2 \frac{C \|\Ga\|_{\infty}\sup_{|v|<\rho}\frac{|\Phi(v)|}{|v|^2} }{E(\lambda,\rho)} \leq \frac{\Lambda_0}{2} < \Lambda_0
\end{align*}
and the proof is completed.
\end{proof}

\section{Existence of a nontrivial homoclinic solution}\label{sect:6}
In this Section, we are going to show the existence of a nontrivial solution to \eqref{eq:hamiltonian}. The main tool will be the abstract theorem \ref{th:abstrExistence}, but we still need a couple of preliminary results: in particular, two concentration-compactness results in the spirit of Lions \cite{Lions1, Lions2} are required, which will allow us to find a nontrivial limit of a minimizing sequence (up to translations), cf. \cite[Lemma 6.1]{BB} and \cite[Proposition 6.2]{BB}.
\begin{Lem}
\label{Lions:vanishing:Lemma}
    Suppose a sequence $(z_n)_n \subset X$ is bounded and for some $R>0$ the following condition
    \begin{equation}
        \label{vanishing:condition}
            \lim_{n \to +\infty} \sup_{y \in \R} \int_{y-R}^{y+R} |z_n|^2 \, dt = 0
    \end{equation}
    holds. Then
    \[
        \lim_{n \to +\infty} \int_{\R} |\Psi(z_n)| \, dt = 0,
    \]
    for any continuous function $\Psi:\R \to \R$ such that
    \[
        \lim_{s \to 0} \frac{\Psi(s)}{s^2} = 0, \quad |\Psi(s)| \lesssim 1 + |s|^r \ \mbox{for all} \ s \in \R,
    \]
    for some $r > 2$.
\end{Lem}

\begin{proof}
    Fix $\Psi : \R \rightarrow \R$ as in the statement. Then, it follows that for every $\varepsilon > 0$ there is $c_\varepsilon > 0$ such that
    $$
    |\Psi(s)| \leq \varepsilon s^2 + c_\varepsilon |s|^r.
    $$
    From the fractional version of Lions' lemma in $H^{1/2} (\R; \R^{2N})$ there follows that $\|z_n\|_r \to 0.$ Since $(z_n)$ is bounded in $X$, and therefore in $L^2(\R; \R^{2N})$, we conclude from the following inequality
    $$
    \limsup_{n \to \infty} \int_\R |\Psi(z_n)| \, dt \leq \limsup_{n \to \infty} \left( \varepsilon \| z_n \|_2^2 + c_\varepsilon \| z_n \|_r^r \right) \lesssim \varepsilon.
    $$
\end{proof}

The above version of the concentration-compactness principle yields the following fact (cf. \cite[Proposition 6.2]{BB}, \cite[Lemma 3.6]{SZ}).
\begin{Prop}
    \label{Lions:nonlinearity:Prop}
    Suppose that a bounded sequence $(z_n)_n \subset X$ satisfies \eqref{vanishing:condition} for some $R>0$. Then,
    \begin{equation}
        \label{Lions:nonliearity:thesis}
            \lim_{n \to +\infty} \int_{\R} \Ga(t)\left(f(z_n) - \lambda g(z_n)\right) \cdot z_n^{\pm} \, dt = 0.
    \end{equation}
\end{Prop}

We are finally ready to prove Theorem \ref{existence:theorem}.
\begin{proof}[Proof of Theorem \ref{existence:theorem}]
    Let us consider the energy functional \eqref{energy:functional} associated to \eqref{eq:hamiltonian}, that is
    \[
        \cJ(z) = \frac12\|z^+\|^2-\frac12\|z^-\|^2 - \int_{\R} \Ga(t) F(z) \, dt + \lambda \int_{\R} \Ga(t) G(z) \, dt.
    \]
    By the computations made in Section \ref{sect:4}, the functional $\cJ$ satisfies the assumptions of Theorem \ref{th:abstrExistence}, hence we get the existence of a Cerami sequence for $\cJ$ such that
    \begin{equation}\label{fromThm}
        \sup_n \cJ(z_n) \lesssim 1, \quad (1+\|z_n\|) \cJ'(z_n) \to 0 \mbox{ in } X^*, \quad \inf_n \triple{z_n} \gtrsim 1,
    \end{equation}
    and, in particular, by Lemma \ref{lem:cerami_bdd}, such a sequence is bounded, so, up to a subsequence, it weakly converges to some $z \in X$. Suppose that
    $$
    \lim_{n\to+\infty} \sup_{y \in \R} \int_{y-1}^{y+1} |z_n|^2 \, dt = 0.
    $$
    By Lemma \ref{vanishing:condition}, applied for $\Psi(s)=|s|^s$, with $s > 2$, we get that $z_n \to 0$ in $L^s(\R;\R^{2N})$, for every $s > 2$. Now, by \eqref{fromThm},
    \begin{align*}
        o(1) 
        & = \cJ'(z_n)(z_n^{\pm}) = (z_n^+,z_n^{\pm}) - (z_n^-,z_n^{\pm}) - \int_{\R} \Ga(t)\left(f(z_n) - \lambda g(z_n)\right) \cdot z_n^{\pm} \, dt\\
        & = 
        (z_n^+,z_n^{\pm}) - (z_n^-,z_n^{\pm}) + o(1),
    \end{align*}
    that leads, respectively, to $\|z_n^+\| \to 0$ and $\|z_n^-\| \to 0$ as $n \to +\infty$. This means that also $\|z_n\| \to 0$ as $n \to +\infty$, and, in particular,
    \[
        \triple{z_n} \leq \|z_n\| \to 0, \text{ as } n \to +\infty.
    \]
    But this contradicts \eqref{fromThm}, therefore, there is a sequence $(y_n)_n \subset \R$ such that, setting $w_n:= z_n(\cdot+y_n)$
    \begin{equation}\label{nonvanishing:w}
        \liminf_{n \to +\infty} \int_{-1}^1 |w_n|^2 \, dt > 0.
    \end{equation}
    Without loss of generality, we may assume that $(y_n)_n \subset \Z$.
    Clearly, $\|w_n\|=\|z_n\|$, so $(w_n)_n$ is bounded in $X$ too, and $w_n \weakto w$ in $X$, for some $w \in X$, with $w \neq 0$ thanks to \eqref{nonvanishing:w}. Moreover,
    \[
        (1+\|w_n\|)\cJ'(w_n) \to 0 \text{ in } X^*.
    \]
    Exploiting now the weak-to-weak* property of $\cJ'$ (cf. (A2)), we get $\cJ'(w_n) \to \cJ'(w)$, with $\cJ'(w)=0$. 
\end{proof}

\section*{Acknowledgements}
The authors would like to thank the anonymous referee for the very valuable remarks and comments.

F. Bernini was partially supported by INdAM-GNAMPA Project 2024 titled \textit{Problemi spettrali e di ottimizzazione di forma: aspetti qualitative e stime quantitative}.
B. Bieganowski and D. Strzelecki were partly supported by National Science Centre, Poland (Grant No. 2022/47/D/ST1/00487).

\bibliography{Bibliography}
\bibliographystyle{abbrv}

\end{document}